\documentclass[11pt,english]{amsart}

\usepackage{amsfonts,amsmath,amssymb,amsthm}
\usepackage[latin1]{inputenc}

\usepackage{graphicx}
\usepackage{float}
\usepackage{tikz}

\newcommand{\Cc}{\mathbb{C}} 

\newcommand{\Rr}{\mathbb{R}}
\newcommand{\Nn}{\mathbb{N}}
\newcommand{\Zz}{\mathbb{Z}}
\newcommand{\Qq}{\mathbb{Q}}
\newcommand{\Ff}{\mathbb{F}}

\renewcommand {\leq}{\leqslant}
\renewcommand {\geq}{\geqslant}
\renewcommand {\le}{\leqslant}
\renewcommand {\ge}{\geqslant}

\newcommand{\supp}{\mathop{\mathrm{supp}}\nolimits}
 
\renewcommand{\sp}{\mathop{\mathrm{sp}}\nolimits}

\newcommand{\point}{^\ast}
\newcommand {\xbar}{\underline{x}}

\newcommand {\tbar}{{\underline{t}}}
\newcommand {\charac}{\mathop{\mathrm{char}}\nolimits}

\newcommand{\defi}[1]{\emph{#1}}

{\theoremstyle{plain}
\newtheorem{theorem}{Theorem}[section]    
\newtheorem{proposition}[theorem]{Proposition}      
\newtheorem{corollary}[theorem]{Corollary}      
}
{\theoremstyle{remark}

\newtheorem{definition}[theorem]{Definition}      
\newtheorem*{remark*}{Remark}  
\newtheorem{remark}[theorem]{Remark}   
\newtheorem{example}[theorem]{Example}
}

\title{Families of polynomials and their specializations}

\author{Arnaud Bodin}
\author{Pierre D\`ebes}
\author{Salah Najib}

\email{Arnaud.Bodin@math.univ-lille1.fr}
\email{Pierre.Debes@math.univ-lille1.fr}
\email{slhnajib@gmail.com}

\address{Laboratoire Paul Painlev\'e, Math\'ematiques, Universit\'e Lille 1, 59655 Villeneuve d'Ascq Cedex, France}
\address{Laboratoire Paul Painlev\'e, Math\'ematiques, Universit\'e Lille 1, 59655 Villeneuve d'Ascq Cedex, France}
\address{Facult\'e Polydisciplinaire de Khouribga, Universit\'e Hassan 1er,
BP 145, Hay Ezzaytoune, 25000 Khouribga, Maroc.}

\subjclass[2010] {Primary 12E05 ; Sec. 11C08, 13P05, 12Y05}

\keywords{Polynomials, Specialization, Irreducibility, Newton Polygon, Good Reduction, Deformation.}

\thanks{{\it Acknowledgment}. This work was supported in part by the Labex CEMPI  (ANR-11-LABX-0007-01)
and by the ANR project ``SUSI'' (ANR-12-JS01-0002-01). The third author wishes to thank Laboratoire Painlev\'e 
of Universit\'e de Lille for its hospitality during several visits in 2014 and 2015.}

\date{\today}

\begin{document}

\begin{abstract}
For a polynomial in several variables depending on some parameters, 
we discuss some results to the effect that for almost all values of the parameters 
the polynomial is irreducible. In particular we recast in this perspective some
results of Grothendieck and of Gao.
\end{abstract}

\maketitle


This paper is devoted to irreducibility questions for families of polynomials in 
several indeterminates $x_1,\ldots,x_\ell$ parametrized by further indeterminates 
$t_1,\ldots,t_s$. 
We assume that  $\ell \geq 2$ and the base field $k$ is algebraically 
closed;  the more arithmetic case $\ell = 1$ depends on the base field and involves 
different tools and techniques. 

Set $\underline t=\{t_1,\ldots,t_s\}$, $\underline x=\{x_1,\ldots,x_\ell\}$ and consider a polynomial 
$F \in k[\underline t, \underline x]$, irreducible in  $\overline{k(\underline t)}[\underline x]$ (where 
$\overline{k(\underline t)}$ is the algebraic closure of $k(\underline t)$); $F$ is 
said to be {\it generically irreducible}. The core question is about the irreducibility of the polynomials obtained by substituting 
elements $t_1^\ast,\ldots,t_s^\ast\in k$ for the corresponding  parameters $t_1,\ldots,t_s$ -- 
the {\it specializations of $F$}. 

More specifically we wish to investigate the following 
problem, as explicitly as possible: 

\noindent
- when the generic irreducibility property is satisfied, show some boundedness results on 
the following set, which we call the {\it spectrum} of $F$:
\vskip 1mm

\centerline{${\rm sp}(F) = \left\{ \underline t^\ast=(t_1^\ast,\ldots,t_s^\ast) \in k^s \mid F(\underline t^\ast,\underline x) 
 \text{ is reducible in } k[\underline x] \right\}$,}
\vskip 1mm

\noindent
and some density results for its complement,

\noindent
- find some criteria for the generic irreduciblity property to be satisfied and 
deduce some new specific examples.

A first approach rests on classical results of Noether and Bertini 
and a second one involves more combinatorial tools like the Newton 
polygon and the associated Minkowski theorem.
We contribute to these approaches by implementing some ideas and results coming 
from connected areas, notably of Grothendieck (Arithmetic Geometry) and 
Gao (Polyhedral Combinatorics).
This leads to new answers to the problem
together with an improved and unified presentation of results from our 
previous papers \cite{BoDeNa2} \cite{BoDeNa1} and other related papers. Those were 
concerned with special cases of the general situation considered here. In particular 
polynomials  $f(x,y)-t$ and variants of those have been much studied and the word 
 ``spectrum'' refers to the classical terminology used in this special case.

\S 1 briefly reviews the classical background and introduces our contribution, 
which is then detailed in \S \ref{sec:grothendieck} and \S \ref{sec:combinatoric}.

%


\section{The classical approaches and our contribution} \label{intro}
\subsection{The arithmetico-geometric approach} \label{ssec:arith-geo-approach}
%


Fix $F \in k[\underline t, \underline x]$ and assume as above that it is generically irreducible.
\subsubsection{Noether} 
Denote by ${\mathcal U}_F$ the open Zariski subset of all $\underline t^\ast\in k^s$ such that 
$\deg(F(\underline t^\ast,\underline x)) = \deg_{\underline x}F(\underline t,\underline x)$. The spectrum ${\rm sp}(F)$ is a proper Zariski 
clo\-sed subset of ${\mathcal U}_F$: there 
exist non-zero polynomials $h_1,\ldots, h_\nu \in k[\underline t]$ \hbox{such that}
\vskip 2mm

\noindent
(1) \hskip 20mm $ {\rm sp}(F) \cap {\mathcal U}_F = {\mathcal Z}(h_1,\ldots, h_\nu) \cap {\mathcal U}_F$
\vskip 1mm

\noindent
where  ${\mathcal Z}(h_1,\ldots, h_\nu)$ denotes the zero set of $h_1,\ldots, h_\nu$.
In other words, for $\underline t^\ast\in k^s$ such that 
$\deg(F(\underline t^\ast,\underline x)) = \deg_{\underline x}F(\underline t,\underline x)$, 
\vskip 2mm

\noindent
{\rm (2)} {\it $F(\underline t^\ast,\underline x)$ 
is reducible in $k[\underline x]$ if and only if $h_m(\underline t^\ast)=0$ \hbox{for each $m=1,\ldots,\nu$.}}

\vskip 2mm

\noindent
This is the classical Noether theorem ({\it e.g.} \cite[\S 3.1 theorem 32]{schinzel-book}) 
which follows from elimination theory. Namely recall that for a given degree $d$ 
and a given number of indeterminates $\ell$, 
if $(a_{\underline i})_{{\underline i}\in I_{\ell,d}}$ are indeterminates that 
correspond to the coefficients of a polynomial of degree $d$ in $\ell$ indeterminates, 
then there exist finitely many homogeneous forms ${\mathcal N}_{j}(a_{\underline i})$ 
($j=1,\ldots, D$) in the $a_{\underline i}$ (${{\underline i}\in I_{\ell,d}}$) 
and with coefficients in $\Zz$ such that: 
\vskip 1mm

\noindent
(3) {\it for a polynomial $P$ of degree $d$ in $\ell$ indeterminates and with coefficients 
$(a_{\underline i}^\ast)_{{\underline i}\in I_{\ell,d}}$ in an algebraically closed field 
$K$ (in such a way that $a_{\underline i}^\ast$ corresponds to $a_{\underline i}$),
the polynomial 
$P$,  if it is of degree $d$, is reducible in $K[\underline x]$ if and only if 
${\mathcal N}_{j}(a_{\underline i}^\ast)=0$, $j=1,\ldots,D$.}

\vskip 2mm

\noindent
Furthermore some subsequent works of Ruppert, Kaltofen and Ch\`eze-Bus\'e-Najib 
provide the following bounds for the degree of the {\it Noether forms ${\mathcal N}_{j}$}: 
\vskip 2mm

\centerline{
$\displaystyle
\left\{ \begin{matrix}
\deg({\mathcal N}_{j}) \leq d^2-1 \hfill & \hbox{\rm if $k$ is of characteristic $0$ \cite{ruppert}}  \hfill\cr
& \hbox{\rm or $p>d(d-1)$}\ \hbox{ \cite[theorem 1]{BCN}} \hfill \cr
\deg({\mathcal N}_{j}) \leq 12 d^6 \hfill & {\rm in\  general} \ 
\hbox{\cite{kaltofen}} \hfill \cr
\end{matrix}
\right.$
}
\vskip 2mm

\noindent
Description (1) of ${\rm sp}(F)$ follows, as explained in \cite[\S 2.3.1]{BoDeNa1}: 
one can take for $h_1,\ldots, h_\nu \in k[\underline t]$ the values of the Noether 
forms at the coefficients in $k[\underline t]$ of the polynomial $F$.
The above bounds yield, in each case
\vskip 2mm

\noindent 
(4) \hskip 10mm $\deg_{t_i}(h_m) \leq \deg_{t_i}(F) \times \displaystyle
\left\{ \begin{matrix}
d^2-1 \hfill \cr
12 d^6 \hfill \cr
\end{matrix}
\right.$ \hskip 5mm ($i=1,\ldots, s$, $m=1,\ldots,\nu$)

\subsubsection{Bertini-Noether} \label{ssec:bertini-noether} 
Every non-zero polynomial $h$ in the ideal $\langle h_1,\ldots, h_\nu\rangle$ of  $k[\underline t]$
has this Bertini-Noether property: for every $\underline t^\ast\in k^s$ such that 
$\deg(F(\underline t^\ast,\underline x)) = \deg_{\underline x}F(\underline t,\underline x)$,

\vskip 2mm

\noindent
{\rm (5)} \hskip 8mm {\it if $h(\underline t^\ast)\not=0$ then $F(\underline t^\ast,\underline x)$
is irreducible in $k[\underline x]$.}

\vskip 2mm

\noindent
Taking for $h$ one of the non-zero polynomials $h_1,\ldots, h_\nu$ yields, 
for $s=1$ and $k$ of characteristic $0$ or $p>d(d-1)$,
\vskip 2mm

\noindent
{\rm (6)} \hskip 18mm ${\rm card}({\rm sp}(F) \cap {\mathcal U}_F) \leq (d^2-1) \deg_{t_1}(F)$
\vskip 2mm

\noindent
More generally, if $s\geq 1$, we have this conclusion:
\vskip 2mm

\noindent
{\rm (7)} {\it For every $i=1,\ldots,s$, the set of $t_i^\ast \in k$
  such that the polynomial $F(t_1,\ldots,t_{i-1},t_i^\ast,t_{i+1},\ldots,t_s,\underline x)$ 
  is of degree $d$ and reducible in the polynomial ring
  $\overline{k(t_1,\ldots,t_{i-1},t_{i+1},\ldots,t_s)}[\underline x]$ is of 
  cardinality $\leq (d^2-1) \deg_{t_i}(F)$.}
  \vskip 2mm

\noindent
Consequently, the polynomial $F(\underline t^\ast,\underline x)$ is irreducible in $k[\underline x]$ provided that
 

\noindent
- $t_1^\ast \in k$ stays out of a certain finite set $E_1$ of cardinality $\leq d^2 \deg_{t_1}(F)$, 
  
\noindent
- $t_2^\ast \in k$ stays out of a certain finite set $E_2$ of cardinality 
$\leq d^2 \deg_{t_2}(F)$ 

\noindent
\hskip 2mm ($E_2$ depending on $t_1^\ast$),

\noindent
- $\cdots$

\noindent
- $t_s^\ast \in k$ stays out of a certain finite set $E_s$ of cardinality $\leq d^2
\deg_{t_s}(F)$  

\noindent
\hskip 2mm ($E_s$ depending on $t_1^\ast,\ldots,t_{s-1}^\ast$). 

\subsubsection{Grothendieck and our contribution}\label{ssec:grothendieck-our-contribution}
We offer an approach in which we replace elimination theory and the Noether theorem
by the Gro\-thendieck good reduction criterion for algebraic covers. We base it on \cite{DeRedSpePol} 
which revisits Grothendieck's work \cite{SGA} \cite{GrMu} with a  
polynomial viewpoint.

For polynomials $F(\underline t, T,Y)$ with $\ell = 2$ indeterminates $T,Y$, monic in $Y$, which this ap\-proach 
is more naturally concerned with, we produce an explicit polynomial ${\mathcal B}_F\in k[\underline t]$, 
called the {\it bad prime divisor} of $F$, that has the Bertini-Noether property (5):
\vskip 2mm

\noindent
{\bf (Corollary \ref{cor:specializations})} {\it Assume ${\rm char}(k)=0$. If $\underline t^\ast\in k^s$ 
satisfies ${\mathcal B}_F(\underline t^\ast) \not= 0$, then the polynomial 
$F(\underline t^\ast,T,Y)\in k[T,Y]$ is irreducible in $k[T,Y]$}.
 \vskip 2mm

The polynomial ${\mathcal B}_F$ is directly computable from the coefficients of $F$ through elementary operations, starting with the discriminant $\Delta_F\in k[\underline t][T]$ of $F$ relative to $Y$ (\S \ref{sec:bad-prime-divisor}).
This general bound for the degree of ${\mathcal B}_{F}$ follows:
\vskip 1mm

\noindent
{\rm (8)} \hskip 18mm $\deg_{t_i}({\mathcal B}_{F}) \leq  16 d^5 \deg_{t_i}({F})$, ($i=1,\ldots,s$).
\vskip 1mm

\noindent
It is not as good as (4); the advantage of ${\mathcal B}_F$ lies in its full explicitness (which may lead to better bounds in specific cases (see \S \ref{ssec:typical})) and in its arithmetic meaning, where the name ``bad prime divisor'' originates: if ${\mathcal B}_F(\underline t^\ast) \not= 0$, the distinct roots (in $\overline{k(\underline t)}$) of $\Delta_F$ remain defined and distinct after specialization of $\underline t$ to $\underline t^\ast\in k^s$. 
The construction improves on 
\cite[\S 3]{BoDeNa2}, which used a result of Zannier rather than
the Grothendieck reduction theory.

We also explain how to get rid of the ``monic'' assumption in corollary \ref{cor:specializations}, to relax
the condition on the characteristic of $k$ and to pass from $2$ to any number $\ell$ of indeterminates.
We finally obtain a statement like corollary \ref{cor:specializations} above but in the bigger generality.  
The polynomial ${\mathcal B}_F$ has to be adjusted but is still explicitly described (see \S \ref{ssec:conjoin} and corollary \ref{cor:s-indetreminates}). 

\subsection{The more combinatorial approach} This second approach uses the Newton representation of 
polynomials as polyhedrons and the associated Minkowski irreducibility criterion. We will also review another 
related approach, based on the Bertini-Krull theorem and compare the two.

\subsubsection{Newton-Minkowski}
The Newton representation identifies monomials $x_1^{i_ 1}\cdots x_\ell^{i_\ell}$ with the corresponding $\ell$-tuples $(i_1,\ldots,i_\ell)\in \Rr^\ell$. Given a polynomial $P\in k[\xbar]$, define then its \defi{support} $\supp(P)$ as the set of monomials appearing in $P$ with a non-zero coefficient and
the \defi{Newton polyhedron} $\Gamma(P)$ of $P$ as the convex closure in $\Rr^\ell$ of $\supp(P)$.
Also define $\Gamma_0(P)$ to be the convex closure of $\supp(P) \cup \{0\}$.
For example, the polynomial
$P(x,y)=x^2y^3+x^3y^2+x^4y^2+x^4y^4+x^5y^3+x^7y^3$ 
has the following Newton polygons $\Gamma(P)$ and $\Gamma_0(P)$ :
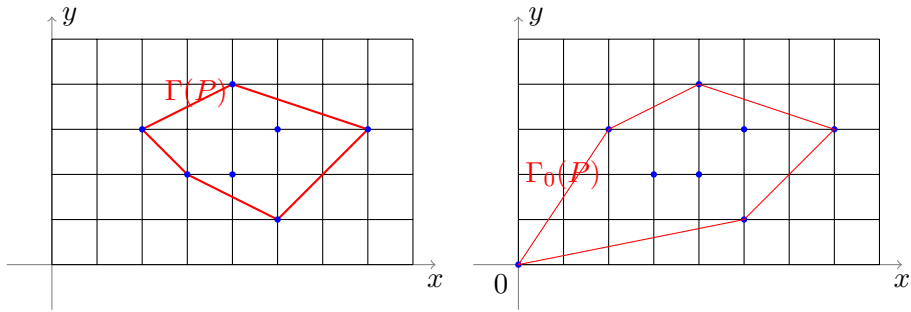
\begin{center}
\begin{tikzpicture}[scale=0.6]

  \draw[gray,->] (-1,0) -- (8.5,0) node[below,black] {$x$};
  \draw[gray,->] (0,-1) -- (0,5.5) node[right,black] {$y$};

 \draw (0,0) grid (8,5);

  \coordinate (O) at (0,0);
  \coordinate (P1) at (2,3);
  \coordinate (P2) at (3,2);
  \coordinate (P3) at (7,3);
  \coordinate (P4) at (4,4);
  \coordinate (P5) at (5,1);
  \coordinate (P6) at (5,3);
  \coordinate (P7) at (4,2);

  \draw[red, thick] (P2)--(P1)--(P4)--(P3)--(P5)--cycle;

     \fill[blue] (P1) circle (2pt);
     \fill[blue] (P2) circle (2pt);
     \fill[blue] (P3) circle (2pt);
     \fill[blue] (P4) circle (2pt);
     \fill[blue] (P5) circle (2pt);
     \fill[blue] (P6) circle (2pt);
     \fill[blue] (P7) circle (2pt);

\node[red] at (3.2,3.8) {$\Gamma(P)$};

%
\end{tikzpicture}
\begin{tikzpicture}[scale=0.6]

  \draw[gray,->] (-1,0) -- (8.5,0) node[below,black] {$x$};
  \draw[gray,->] (0,-1) -- (0,5.5) node[right,black] {$y$};

 \draw (0,0) grid (8,5);

  \coordinate (O) at (0,0);
  \coordinate (P1) at (2,3);
  \coordinate (P2) at (3,2);
  \coordinate (P3) at (7,3);
  \coordinate (P4) at (4,4);
  \coordinate (P5) at (5,1);
  \coordinate (P6) at (5,3);
  \coordinate (P7) at (4,2);


     \fill[blue] (P1) circle (2pt);
     \fill[blue] (P2) circle (2pt);
     \fill[blue] (P3) circle (2pt);
     \fill[blue] (P4) circle (2pt);
     \fill[blue] (P5) circle (2pt);
     \fill[blue] (P6) circle (2pt);
     \fill[blue] (P7) circle (2pt);


 \fill[blue] (O) circle (2pt);
\node[below left] at (O) {$0$};

 \draw[red] (O)--(P1)--(P4)--(P3)--(P5)--cycle;
\node[red] at (1,2) {$\Gamma_0(P)$};
\end{tikzpicture}  
\end{center}

The Minkowski theorem is this irreducibility criterion, where the sum $A+B$ of the two subsets $A$ and $B$ of $\Rr^\ell$ is $A+B = \{a+b \mid a\in A, b\in B\}$:
\vskip 2mm

\noindent
(10) {\it If $P = P_1\cdot P_2$, with $P_1,P_2\in k[\underline x]$, then $\Gamma(P) = \Gamma(P_1) + \Gamma(P_2)$. Consequently, if $\Gamma(P)$ is not summable, then $P$ is irreducible in $k[\underline x]$,}
\vskip 2mm

\noindent
where we say that $\Gamma(P)$ is \defi{summable} if it writes as the sum $A+B$
of two convex subsets $A$ and $B$ with integral vertices (in $\Nn^\ell$), each of 
them having at least two points. The converse is false: $(x-y+1)(x+y+1)+1$ is 
irreducible but its Newton polygon is summable as it is also that of $(x-y+1)(x+y+1)$. 




\subsubsection{Our perspective} The Minkowski theorem can be viewed as follows. Suppose given a non summable convex
subset $\Gamma \subset \Rr^\ell$ with a set ${\mathcal V}_\Gamma$ of vertices contained in $\Nn^\ell$.
Consider the polynomial obtained by summing all the monomials $t_{\underline i} \hskip 2pt x_1^{i_1}\cdots x_\ell^{i_\ell}$ where 
${\underline i} = (i_1,\ldots,i_\ell)$ ranges over all the set ${\mathcal I}_\Gamma$ of all monomials inside $\Gamma$ and the corresponding $t_{\underline i}$ are indeterminates, forming a set $\underline t_\Gamma$. Denote this polynomial by $F_\Gamma$, which is in $k[\underline t_\Gamma,\underline x]$. 
\vskip 1mm

\noindent
(11) {\it $F_\Gamma$ is generically irreducible and even satisfies this stronger irreducibility property: 
$F_\Gamma(\underline t_\Gamma^\ast,\underline x)$ is irreducible in $k[\underline x]$, for every specialization $\underline t_\Gamma^\ast$ such that 
$t_{\underline i}^\ast \not=0$ if $\underline i \in {\mathcal V}_\Gamma$.}
\vskip 1mm

\noindent
The condition on the specialization $\underline t_\Gamma^\ast$ indeed assures that the Newton polyhedron of $F_\Gamma(\underline t_\Gamma^\ast,\underline x)$ is $\Gamma$. 
This yields this combinatorial version of the Bertini-Noether conclusion.

\begin{proposition}
Let $F(\tbar,\xbar) \in k[\tbar,\xbar]$ be a polynomial with Newton polyhedron $\Gamma$ 
{\rm (}as a polynomial in $\xbar${\rm )}. Denote its coefficients by $h_{\underline i}(\underline t)$ {\rm (}${\underline i}\in {\mathcal I}_\Gamma${\rm )}. If $\Gamma$ is  not summable then $F(\tbar\point,\xbar)$ is irreducible in $k[\xbar]$ for every specialization $\tbar\point$ such that 
$h_{\underline i}(\tbar\point)\not=0$ for each $\underline i \in {\mathcal V}_\Gamma$. Consequently, inequalities {\rm (6)}, 
{\rm (7)} from \S \ref{ssec:bertini-noether} hold with $(d^2-1) \deg_{t_i}(F)$ replaced by ${\rm card}({\mathcal V}_\Gamma) \deg_{t_i}(F)$.
\end{proposition}


%
%

\subsubsection{Bertini-Krull, Gao and our contribution} \label{ssec:bertini-krull-gao}
There exist efficient criteria and algorithms to decide whether a given convex set is 
summable, notably in a series of papers by Gao {\it et al} \cite{gao} \cite{GaoLau} \cite{ASGaoLau}. 
We 
use them in \S \ref{sec:monomial} to produce some new classes of generically irreducible polynomials.
%
%
%
%
%

The polynomial $F_\Gamma$ from (11) is linear in the parameters $t_{\underline i}$ and we will 
focus on this special situation. That is: we will assume $F$ is of the form
\vskip 1mm

\noindent
(13) \hskip 25mm $F = P(\underline x) -t_1Q_1(\underline x) - \cdots - t_s Q_s(\underline x)$
\vskip 1mm

\noindent
with $P,Q_1,\ldots,Q_s\in k[\underline x]$, which can viewed as a {\it deformation} of the polynomial $P$
by the polynomials $Q_1,\ldots,Q_s$. Here is an example of a result that can be deduced 
from Gao's criteria.
\vskip 2mm

\noindent
{\bf (Corollary \ref{cor:Gao})}
{\it Let $P, Q \in k[\xbar]$ such that 
\vskip 0,5mm

\noindent
{\rm ($\alpha$)} $\Gamma(P)$ is contained in a  hyperplane $H\subset \Rr^\ell$ not passing through the origin,
\vskip 0,5mm

\noindent
{\rm ($\beta$)} the coordinates of all the vertices of $\Gamma(P)$ are relatively prime,
\vskip 0,5mm

\noindent
{\rm ($\gamma$)} $Q(0,\ldots,0)\not=0$, $\Gamma(Q)\subset \Gamma_0(P)$ and no monomial of $Q$ is \hbox{a vertex of $\Gamma(P)$.}
\vskip 1mm

\noindent
Then the polynomial $F=P-tQ$ is generically irreducible and even has this stronger property: $P-t^\ast Q$ is irreducible in $k[\xbar]$ for every $t^\ast\in k\setminus\{0\}$.}
\vskip 2mm

\noindent
For example, if $p,q,r$ are $3$ relatively prime positive integers, the polynomial
\vskip 2mm

\centerline{$F(t,x,y,z) = x^p+y^q + z^r + t\hskip 2pt (\hskip 2pt \sum_{\frac{i}{p}+\frac{j}{q}+\frac{k}{r}<1} a_{i,j,k} x^i y^j z^k$\hskip 2pt  )\hskip 2mm with $a_{0,0,0}\not=0$}
\vskip 2mm

\noindent
satisfies the conclusion of corollary \ref{cor:Gao} (as shown in example \ref{example:Gao2-1}, condition $a_{0,0,0}\not=0$ can in fact be removed).
Note further that the assumptions on $P$, $Q$ in corollary \ref{cor:Gao} only depend on the Newton polyhedrons $\Gamma(P)$ and $\Gamma(Q)$.

Before getting to applications of Gao's results, we review in \S \ref{ssec:bertini-krull} a more classical 
approach for polynomials  as in (13), based on the Bertini-Krull theorem. The special case $P-tQ$ has been much studied due to its connection with the indecomposability and the spectrum of the rational function $P/Q$. The even more special case $Q=1$ is of particular interest since a famous theorem of Stein provides an 
optimal bound for the cardinality of ${\rm sp}(P-t)$ which is sharper than the 
Bertini-Noether bound. This bound issue leads us to discuss to what extent 
the spectrum of a rational function can be prescribed. 
Finally we review and compare with Gao's results some results of \cite{BoDeNa1} 
concerned with the special case of (13) that $Q_1,\ldots,Q_s$ are monomials, for which the Bertini-Krull theorem is also a main ingredient.

%


%

\section{The Grothendieck arithmetico-geometric approach} \label{sec:grothendieck}

This section elaborates on \S \ref{ssec:grothendieck-our-contribution}. \S \ref{ssec:reduction_l=2} 
explains the reduction to the situation of $\ell=2$ indeterminates.  \S \ref{sec:bad-prime-divisor} 
introduces the bad prime divisor and the Grothendieck approach. Finally, \S \ref{ssec:conjoin} conjoins 
\S \ref{ssec:reduction_l=2} and \S \ref{sec:bad-prime-divisor}.

The general notation introduced in \S \ref{intro} is retained.

\subsection{Reduction to the situation $\ell=2$} \label{ssec:reduction_l=2} The main result of 
this subsection is the following statement. For more generality, several polynomials $F_1,\ldots,F_h$
replace the single polynomial $F$ from \S \ref{intro}. 

The following additional notation is needed.
Given a polynomial $P$ in the indeterminates $\underline y= (y_1,\ldots, y_N)$ 
with coefficients in some integral domain and $B=(b_{ij})_{1\leq i,j\leq N}$ a $N\times N$-matrix with entries in the same domain, set
\vskip 2mm
\centerline{$\displaystyle P(B\cdot \underline y) = P(\sum_{j=1}^N b_{1j} y_j, 
\ldots, \sum_{j=1}^N b_{Nj} y_j)$}
\vskip 2mm



\begin{theorem} \label{cor:reduction} Let $F_1,\ldots, F_h \in k[\underline t,\underline x]$, assumed to be
irreducible in $\overline{k(\underline t)}[\underline x]$ 
and $\kappa \subset k$ be an infinite subfield. There is a matrix $B=(b_{ij})_{i,j}\in {\rm GL}_{\ell}(\kappa)$ such that the polynomial 
$F_i(\underline t, B\cdot \underline x)$, which is of the form
\vskip 0mm

\centerline{$\displaystyle F_i(t_1,\ldots,t_s, \sum_{j=1}^\ell b_{1j} x_j, \ldots, 
\sum_{j=1}^\ell b_{\ell j} x_j)$\hskip 5mm {\rm (}$i=1,\ldots, h${\rm )} }

\vskip 1mm

\noindent
is irreducible in $\overline{k(\underline t,x_1,\ldots,x_{\ell-2})}[x_{\ell-1},x_{\ell}]$ 
and satisfies the degree condition
 $\deg_{x_{\ell-1},x_{\ell}}(F_i(\underline t, B\cdot \underline x)) 
 = \deg_{\underline x}(F_i(\underline t,\underline x))$.
\end{theorem}

\begin{remark} \label{remark:reduction} If $F_1,\ldots, F_h$ are only irreducible in $k(\underline t)[\underline x]$, 
the same conclusion holds with these adjustments: $B$ should be a matrix 
$B=(b_{ij})_{i,j}\in {\rm GL}_{s+\ell}(\kappa)$ that applies to the $s+\ell$ indeterminates $t_1,\ldots,t_s,x_1,\ldots,x_\ell$; the resulting polynomial 
$F_i(B\cdot (\underline t, \underline x))$ is of the form
\vskip 1,5mm

\centerline{$\displaystyle F_i(\hskip 1pt \beta_1(\underline x) \hskip -2pt + \hskip -2pt \tau_1(\underline t)\hskip 1pt, \hskip 1pt \ldots, 
\hskip 1pt \beta_{s+\ell}(\underline x)\hskip -2pt +\hskip -2pt \tau_{s+\ell}(\underline t)\hskip 1pt)$\hskip 5mm ($i=1,\ldots, h$)}
\vskip 1,5mm


\noindent
\hbox{for some $\kappa$-linear forms $\beta_1(\underline x), \ldots, \beta_{s+\ell}(\underline x)\in \kappa[\underline x]$ 
and $\tau_1(\underline t),\ldots,\tau_{s+\ell}(\underline t)
\in \kappa[\underline t]$. }
%
\end{remark}


The main tool in the proof of theorem \ref{cor:reduction} is proposition \ref{prop:reduction} below. 
Let $A$ be an integral domain with fraction field $K$ and let $\underline y= \{y_1,\ldots, y_N\}$, 
$\underline z = \{z_1,\ldots,z_M\}$ be two sets of indeterminates with $N\geq 2$, $M\geq 1$. 

\begin{proposition} \label{prop:reduction} Let 
$\kappa \subset A$ be an infinite subfield and $P_1,\ldots,P_h \in 
A[\underline z, \underline y]$ be $h$ polynomials, irreducible in 
$\overline K[\underline z, \underline y]$. There exists a matrix 
$B=(b_{ij})_{i,j}\in {\rm GL}_{M+N}(\kappa)$ such that  for $i=1,\ldots, h$,
\vskip 1mm

\noindent 
{\rm (a)} the polynomial $P_i(B\cdot (\underline z, \underline y))$, 
which is of the form
\vskip 1,5mm

\centerline{$\displaystyle P_i\hskip 1pt (\hskip 1pt \beta_1(\underline y) \hskip -2pt + \hskip -2pt \tau_1(\underline z)\hskip 1pt, \hskip 1pt \ldots, 
\hskip 1pt \beta_{M+N}(\underline y)\hskip -2pt +\hskip -2pt \tau_{M+N}(\underline z)\hskip 1pt)$ }
\vskip 1,5mm


\hskip 16mm for some $\kappa$-linear forms $\beta_1(\underline y), \ldots, \beta_{M+N}(\underline y)\in \kappa[\underline y]$ 

\hskip 16mm and some $\kappa$-linear forms 
$\tau_1(\underline z),\ldots,\tau_{M+N}(\underline z)
\in \kappa[\underline z]$,

\vskip 3mm

\noindent
is irreducible in $\overline{K(\underline z)}[\underline y]$,
\vskip 2mm

\noindent 
{\rm (b)} furthermore $\deg_{\underline y}(P_i(B\cdot (\underline z, \underline y))) 
= \deg_{\underline z,\underline y}(P_i)$.
\end{proposition}

\begin{remark}
There are several variants of proposition \ref{prop:reduction} in the literature: \cite[ch5, theorem 3d]{schmidt}
for $1$ polynomial ($h=1$) and $1$ parameter ($M=1$), \cite[lemma 7]{kaltofen} for $h=1$, \cite[proposition 1]{salah_lorenzini} for $M=1$.
Our version has several polynomials and several parameters and the produced 
matrix has coefficients in any given infinite subfield of the ring $A$. 
\end{remark}

Theorem \ref{cor:reduction} corresponds to the following special case of proposition 
\ref{prop:reduction}: $\underline z=(x_1,\ldots,x_{\ell-2})$, 
$\underline y = (x_{\ell-1}, x_\ell)$ and $A=k[\underline t]$ while remark \ref{remark:reduction} 
corresponds to the special case:
$\underline z=(t_1,\ldots,t_s, x_1,\ldots,x_{\ell-2})$, 
$\underline y = (x_{\ell-1}, x_\ell)$ and $A=k$.

\begin{proof}[Proof of proposition \ref{prop:reduction}]
Proposition \ref{prop:reduction} is a generalization of \cite[proposition 1]{salah_lorenzini}, 
which corresponds to the special situation: $\underline z=z_1$ and $\kappa=A=K$.

First we generalize \cite[proposition 1]{salah_lorenzini} to the situation 
``$\kappa \subset A$ infinite'' (but still with  $\underline z=z_1$). 
This only requires to adjust the proof of \cite{salah_lorenzini}: the matrices 
that are constructed there with coefficients in $K$ can be chosen with 
coefficients in $\kappa$, the main point being that $\kappa$ is infinite. 
The core of the proof is the Matsusaka-Zariski theorem \cite[proposition 10.5.2]{FrJa}.

This generalized \cite[proposition 1]{salah_lorenzini}, applied to the situation of 
proposition \ref{prop:reduction}, provides a matrix $B_1\in {\rm GL}_{M+N}(\kappa)$ 
such that $P_i(B_1\cdot (\underline z, \underline y))$ is irreducible 
in $\overline{k(z_1)}[z_2,\ldots,z_M,y_1,\ldots,y_N]$ and satisfies the degree condition 
\vskip 1mm

\centerline{$\deg_{z_2,\ldots,z_M,\underline y}(P_i(B_1\cdot (\underline z, \underline y))) 
= \deg_{\underline z, \underline y}(P_i)$ ($i=1,\ldots, h$).} 

\vskip 1mm

\noindent
Apply next the generalized \cite[proposition 1]{salah_lorenzini} to the 
polynomials $P_i(B_1\cdot (\underline z, \underline y))$, $i=1,\ldots, h$, 
viewed as polynomials in the indeterminates $z_2,\ldots,z_M,\underline y$ 
and to the same infinite 
subfield $\kappa$ (of the coefficient field $k(z_1)$ of these polynomials). 
This  provides a matrix $B_2\in 
{\rm GL}_{M+N-1}(\kappa)$ which we make a matrix $B_2\in 
{\rm GL}_{M+N}(\kappa)$ by letting it be the identity on the missing coordinate 
$z_1$ and is such that the polynomial $P_i(B_1 B_2 \cdot (\underline z, \underline y))$ 
is irreducible in $\overline{k(z_1,z_2)}[z_3,\ldots,z_M,\underline y]$ and satisfies 
$\deg_{z_3,\ldots,z_M,\underline y}(P_i(B_1 B_2 \cdot (\underline z, \underline y))) 
= \deg_{\underline z, \underline y}(P_i)$ ($i=1,\ldots, h$). Iterating this process 
leads to the desired statement. \end{proof}

\subsection{The bad prime divisor and the Grothendieck approach}  
\label{sec:bad-prime-divisor} 
This subsection is based on \cite{DeRedSpePol}. 
The context there is that of polynomials $F\in A[T,Y]$ with coefficients 
in a Dedekind domain $A$, with $A=\Zz$ and $A=k[t]$ as typical examples. Here we
focus on the situation $A=k[\underline t]$, which is not a Dedekind domain if $s\geq 2$. 
We can however specialize one by one the parameters $t_1,\ldots,t_s$ so as to work at each step with the ring $A= k(t_1,\ldots,t_i)[t_{i+1}]$ which is a Dedekind domain (see \S \ref{ssec:successive_specializations}).
\vskip 1mm

Let $A$ be an integral domain with fraction field $K$ and $F\in A[T,Y]$ be a polynomial, 
irreducible in $\overline K[T,Y]$. Up to switching $Y$ and $T$, one may assume that 
$n=\deg_Y(F)\geq 1$.
Assume further $n\geq 2$; the remaining case $n=1$ is trivial (see remark \ref{rem:n=1}).
Set $m=\deg_T(F)$; $m\geq 1$. Assume $A$ is of characteristic $0$ or $p>(2n^2-n) m$. 
Paragraphs \S \ref{ssec:monic_reduction}-\ref{ssec:central-result} recall from \cite{DeRedSpePol}
the construction of the bad prime divisor and its Bertini-Noether property. 

\subsubsection{Preliminary reduction to a monic polynomial} \label{ssec:monic_reduction}  
First reduce to the situation where $F$ is monic in $Y$ by replacing 
\vskip 1mm
\centerline{$F(T,Y)=F_0 Y^n+ F_1 Y^{n-1} + \cdots + F_n$}
\vskip 1mm  
\centerline{with $F_0,F_1,\ldots, F_n \in A[T]$.}
\vskip 1mm  

\centerline{{\rm by}}
\centerline{$\displaystyle Q(T,Y)=F_0^{n-1} F(T,\frac{Y}{F_0})=Y^n+ F_1 Y^{n-1} +\cdots + F_0^{n-1}F_n$}
\vskip 1mm

\noindent
We have $\deg_Y(Q) =n$ and $\deg_T(Q)\leq nm$, so $p> (2\deg_Y(Q)-1)\deg_T(Q)$
in the case $p>0$. Consequently the polynomial $Q$, as a polynomial in $Y$, has only simple 
roots in $\overline{K(T)}$ 
and so do the irreducible factors in $K[T]$ of its discriminant w.r.t. $Y$, 
as polynomials in $T$; they have only simple roots in $\overline K$.
This is a starting hypothesis in \cite{DeRedSpePol}.

\subsubsection{Definition of the bad prime divisor} \label{ssec:bad_prime-divisor}  
Assume from now on, in addition to $F\in A[T,Y]$, irreducible in $\overline K[T,Y]$, that $F$ is monic in $Y$, $n=\deg_Y(F)\geq 2$, 
$m=\deg_T(F)\geq 1$, that the characteristic of $A$ is $0$ or $p> (2n-1)m$ and that $A$ is integrally closed.

Denote the {\it discriminant} of $F$ relative to $Y$ by
\vskip 1mm
\centerline{$\Delta_F = {\rm disc}_Y(F)$} 
\vskip 1mm

\noindent
We have 
$\Delta_F\in A[T]$ and $\Delta_F\not=0$. Consider the {\it reduced discriminant}: 
\vskip 2mm

\centerline{$\Delta_F^{\rm red} = (\Delta_{F,0})^\rho \prod_{i=1}^\rho (T-\tau_i)$}
\vskip 2mm

\noindent
where $\Delta_{F,0}$ is the leading coefficient of $\Delta_F$ and $\tau_1,\ldots,\tau_\rho$ 
are the distinct roots of $\Delta_F$ in $\overline K$. From \cite[lemma 2.1]{DeRedSpePol}, 
we have $\Delta_F^{\rm red} \in A[T]$ and

\vskip 2mm

\centerline{$\displaystyle \Delta_F^{\rm red} = \hskip 1mm \Delta_{F,0}^{\rho-1} 
\hskip 1mm\frac{\Delta_F}{\hbox{\rm gcd}(\Delta_F,\Delta_F^\prime)}$}
\vskip 2mm

\noindent
where the {\rm gcd} is calculated in the ring $K[T]$ and made to be monic by multiplying by the suitable non-zero constant. Furthermore the discriminant
\vskip 2mm

\centerline{$\displaystyle \hbox{\rm disc}(\Delta^{\rm red}_F)
= (\Delta_{F,0})^{2\rho(\rho-1)} \prod_{1\leq i\not=j\leq \rho} (\tau_j-\tau_i)$}
\vskip 1mm

\noindent
is an element of $A$ and is non-zero as by construction $\Delta^{\rm red}_F(T)$ 
has no multiple root in $\overline K$. Define then an element ${\mathcal B}_F$ by
\vskip 1mm

\centerline{$\displaystyle {\mathcal B}_F = \Delta_{F,0} \cdot \hbox{\rm disc}(\Delta^{\rm red}_F)$}
\vskip 1mm

\noindent
We have ${\mathcal B}_F\in A$ and ${\mathcal B}_F\not=0$.

\begin{definition}
The maximal ideals ${\mathfrak p}\subset A$ that contain ${\mathcal B}_F$ are 
called the {\it bad primes} of $F\in A[T,Y]$ and ${\mathcal B}_F$ is called the {\it bad prime 
divisor}. Maximal ideals ${\mathfrak p}\subset A$ that are not bad are said to be {\it good}.
\end{definition}

\begin{remark}\label{rem:n=1} In the case $\deg_Y(P)=1$, $\deg_T(P)\leq 1$, the construction
leads to ${\mathcal B}_F=1$. All maximal ideals ${\mathfrak p}\subset A$ are good and the main result,
theorem  \ref{thm:reduction_property} below, trivially holds.
 \end{remark}

%
%


\subsubsection{The main result} \label{ssec:central-result}
In addition to the assumptions of \S \ref{ssec:bad_prime-divisor}, assume that $A$ is a Dedekind domain. Let ${\mathcal G}$ be the Galois group of the 
splitting field of $F$ over $K(T)$.

If ${\mathfrak p}\subset A$ is a prime ideal, denote the residue field $A/{\mathfrak p}$ 
by $\kappa_{\mathfrak p}$, the reduction map by $s_{\mathfrak p}: A\rightarrow \kappa_{\mathfrak p}$, 
the localized ring of $A$ by ${\mathfrak p}$ by $A_{\mathfrak p}$ and the polynomial obtained 
by reducing the coefficients of $P$  by $s_{\mathfrak p}(P)$.

\begin{theorem}[theorem 2.6 of \cite{DeRedSpePol}] \label{thm:reduction_property} 
Let ${\mathfrak p}\subset A$ be a good prime of $F$ such that
 $|{\mathcal G}|\notin {\mathfrak p}$.  Then we have these two conclusions:
\vskip 1mm

\noindent
{\rm (Good Behaviour)} We have $\displaystyle {\mathcal B}_{s_{\mathfrak p}(F)} 
= s_{\mathfrak p}({\mathcal B}_F)\not=0$.
\vskip 1mm

%

\noindent
{\rm (Good Reduction)}  The polynomial $s_{\mathfrak p}(F)$ 
is ir\-re\-du\-cible in $\overline{\kappa_{\mathfrak p}}[T,Y]$.
\end{theorem}

Condition ${\mathcal B}_{s_{\mathfrak p}(F)} = s_{\mathfrak p}({\mathcal B}_F)\not=0$ 
 rephrases as saying that no distinct roots $\tau_i$ and $\tau_j$ of 
$\Delta_F$ meet modulo ${\mathfrak p}$ 
and none of the roots $\tau_i$ meets $\infty$ modulo ${\mathfrak p}$.


  \subsubsection{Specializations in families of polynomials} \label{ssec:successive_specializations} 
Take $A=k[\underline t]$ and consider a polynomial $F\in k[\underline t][T,Y]$ as in \S \ref{ssec:bad_prime-divisor}.
The bad prime divisor ${\mathcal B}_F$ is an element of $k[\underline t]$ and the bad primes are the 
$s$-tuples $\underline t^\ast = (t_1^\ast,\ldots,t_s^\ast)$ such that ${\mathcal B}_F(\underline t^\ast)=0$. 

Theorem \ref{thm:reduction_property} cannot be applied directly if $s\geq 2$ as $A=k[\underline t]$ 
is not a Dedekind domain but can be applied to $F$ viewed in $k(t_1,\ldots,t_{s-1})[t_s][T,Y]$. 
 It is readily checked that the bad prime divisor relative to $k(t_1,\ldots,t_{s-1})[t_s]$ is the same as 
 relative to the smaller ring $k[t_1,\ldots,t_s]$. Hence it 
 is  the polynomial ${\mathcal B}_F$ in $k[t_1,\ldots,t_s]$ introduced above. 
 
From the assumptions, $k$ is of characteristic $0$ or $p>(2n-1)m \geq n$. Therefore, as $|{\mathcal G}|$
divides $n!$, $p$ cannot divide $|{\mathcal G}|$ and $|{\mathcal G}|$ is in no prime ideal 
${\mathfrak p}$ of $k[\underline t]$. Let $t_s^\ast\in k$ such that 
${\mathcal B}_F(t_1,\ldots,t_{s-1},t_s^\ast) \not= 0$. From theorem \ref{thm:reduction_property}, 
$F(t_1,\ldots,t_{s-1},t_s^\ast,T,Y)$ is irreducible in 
$\overline{k(t_1,\ldots,t_{s-1})}[T,Y]$ and its bad prime divisor is 
${\mathcal B}_F(t_1,\ldots,t_{s-1},t_s^\ast) \in k[t_1,\ldots,t_{s-1}]$. 
Theorem \ref{thm:reduction_property} can then be applied to $F(t_1,\ldots,t_{s-1},t_s^\ast,T,Y)$ 
to specialize $t_{s-1}$. An inductive argument finally leads to this conclusion: 

\begin{corollary} \label{cor:specializations} If $(t_1^\ast,\ldots,t_s^\ast)\in k^s$ 
satisfies ${\mathcal B}_F(t_1^\ast,\ldots ,t_s^\ast) \not= 0$, then the polynomial 
$F(t_1^\ast,\ldots ,t_s^\ast,T,Y)\in k[T,Y]$ is irreducible in $k[T,Y]$.
\end{corollary} 


\subsubsection{Reduction modulo $p$} The unifying context ``$F\in A[T,Y]$ with $A$ 
is a Dedekind domain'' also 
allows the special case $F\in \Zz[T,Y]$ and the prime ${\mathfrak p}$ is a 
prime number $p$. In this situation 
the bad prime divisor ${\mathcal B}_F$ is a non-zero integer, 
the bad primes are the prime numbers  dividing ${\mathcal B}_F$ and 
theorem \ref{thm:reduction_property} yields this effective version of Ostrowski's theorem:

\begin{corollary} If $p$ is a prime number not dividing ${\mathcal B}_F$ nor 
$|{\mathcal G}|$, then the reduced polynomial $\overline F$ is irreducible 
in $\overline{\Ff_p}[T,Y]$.
\end{corollary}

\subsubsection{Explicitness of ${\mathcal B}_F$} \label{ssec:typical}
In the two typical situations $A=k[\underline t]$ and $A=\Zz$, one can explicitly bound the height and 
the partial degrees of ${\mathcal B}_F$. However as already alluded to in \S \ref{ssec:grothendieck-our-contribution}, the bounds are big and do not improve on previously known ones; we refer to \cite[\S 4.2]{DeRedSpePol} for the estimate
$\deg_{t_i}({\mathcal B}_{F}) \leq  16 d^5 \deg_{t_i}({F})$ given in 
\S \ref{ssec:grothendieck-our-contribution}. 

On the other hand, one can compute the exact value of ${\mathcal B}_F$ for specific polynomials {\it via} a simple computer program. For example, for $F(T,Y) = Y^3 - 6T^2 + TY - 2$ (in case $A=\Zz$), one obtains ${\mathcal B}_{F}=2^{16} \cdot 3^{19} \cdot 431 \cdot 433$. It can be checked that $F$ is irreducible modulo $5$, as expected, that $F$ is reducible modulo the prime $2$, which divides ${\mathcal B}_{F}$, and is irreducible modulo $3$ although $3$ 
divides ${\mathcal B}_{F}$; the bad prime divisor is not optimal. Similarly for $F(T,Y) = T^2 Y + 2TY^2 + Y^3 + 3T + 3Y + t$ (in case $A=\Qq[t]$), one obtains ${\mathcal B}_{F}= t^2 \cdot (t - 1)^3 \cdot (t + 1)^3$. It can be checked that $F(2,T,Y)$ is irreducible, $F(0,T,Y)$ is reducible and $F(1,T,Y)$ is irreducible.

\subsection{Conjoining \S \ref{ssec:reduction_l=2} and \S \ref{sec:bad-prime-divisor}}  
\label{ssec:conjoin} 
In our original situation, we have a polynomial 
$F\in k[\underline t,\underline x]$ assumed to be irreducible in 
$\overline{k(\underline t)}[\underline x]$. Assume further that $k$ is of characteristic 
$0$ or $p>2 \deg(F) ^3$. 


\begin{corollary}\label{cor:s-indetreminates} There is a non-zero polynomial $\widetilde{\mathcal B}_F(\underline t,\underline x) \in k[\underline t,\underline x]$, explicitly constructed in the proof, with the following property. For every $\underline t^\ast\in k^s$ such that 
$\widetilde{\mathcal B}_F(\underline t^\ast, \underline x)\not= 0$ 
{\rm (in $k[\underline x]$)}, the polynomial $F(\underline t^\ast,\underline x)$ 
is irreducible in $k[\underline x]$.
\end{corollary}

\begin{proof}
The number of indeterminates $x_1,\ldots,x_\ell$ being $\ell \geq 2$, we 
may assume $\deg_{x_j}(F) \geq 1$ for $j=\ell-1,\ell$. Set $T=x_{\ell-1}$, 
$Y=x_{\ell}$, $A= k[\underline t,x_1,\ldots,x_{\ell-2}]$ and $K={\rm Frac}(A)$. 
From theorem \ref{cor:reduction}, there is a matrix $B \in {\rm GL}_\ell(k)$ 
such that the polynomial 
\vskip 1mm

\centerline{$F(\underline t, B\cdot \underline x)$}
\vskip 1mm

\noindent
is in $A[T,Y]$ and is irreducible in $\overline K[T,Y]$. The assumption on the characteristic of $k$ guarantees that
the one made on the characteristic of $A$ in \S \ref{sec:bad-prime-divisor} is satisfied. Apply \S \ref{ssec:monic_reduction} 
to make $F$ monic in $Y$.  Denote then its bad prime divisor 
by $\widetilde{\mathcal B}_F$; it is a non-zero element of $k[\underline t,x_1,\ldots,x_{\ell-2}] \subset k[\underline t,\underline x]$. 
\vskip 2pt

Let $\underline t^\ast\in k^s$ such that $\widetilde{\mathcal B}_F(\underline t^\ast,\underline x)\not=0$ 
(in $k[\underline x]$). The set of $(\ell-2)$-tuples $(x_1^\ast,\ldots,x_{\ell-2}^\ast)\in k^{\ell-2}$ 
such that $\widetilde{\mathcal B}_F(\underline t^\ast,x_1^\ast,\ldots,x_{\ell-2}^\ast) = 0$ is 
a proper Zariski closed subset ${\mathcal Z} \subset k^{\ell-2}$. From corollary 
\ref{cor:specializations},
for every 
$(x_1^\ast,\ldots,x_{\ell-2}^\ast)\in k^{\ell-2}\setminus {\mathcal Z}$, 
the polynomial obtained from $F(\underline t, B\cdot \underline x)$ by 
specializing $\underline t$ to $\underline t^\ast$ and $x_k$ to $x_k^\ast$ for 
$k=1,\ldots, \ell-2$, is irreducible in  $k[T,Y]$. {\it A fortiori}
the polynomial obtained by only 
specializing $\underline t$ to $\underline t^\ast$ is irreducible in $k[\underline x]$. This 
polynomial is $F(\underline t^\ast, B \cdot \underline x)$. The result follows as the 
matrix $B$ is invertible.
\end{proof}

Thanks to the generality of theorem  \ref{cor:reduction}, corollary \ref{cor:s-indetreminates} extends
to the situation of several polynomials $F_1,\ldots,F_h$: $\widetilde{\mathcal B}_F$ should be replaced by
the product $\widetilde{\mathcal B}_{F_1} \cdots \widetilde{\mathcal B}_{F_h}$ for some matrix $B$ working for all  
$F_1,\ldots,F_h$; the two indeterminates $x_i$ which play the role of $T$ and $Y$ may however differ.

\section{Combinatorial approach to irreducibility criteria} \label{sec:combinatoric}

The central aim of this section is to provide some generic irreducibility criteria; it elaborates 
on \S \ref{ssec:bertini-krull-gao}. \S \ref{ssec:bertini-krull} first reviews the approach based 
on the Bertini-Krull theorem. \S \ref{sec:monomial} is devoted to
applications of the more recent Gao criteria for non-summability of polyhedrons.

In this section we assume that $F$ is of the form
\vskip 1mm

\noindent
 \hskip 25mm $F(\underline t,\underline x) = P(\underline x) -t_1Q_1(\underline x) - \cdots - t_s Q_s(\underline x)$
\vskip 1mm

\noindent
that is, is a linear deformation of the polynomial $P\in k[\underline x]$
by the polynomials $Q_1,\ldots,Q_s\in k[\underline x]$.

\subsection{The Bertini-Krull approach} \label{ssec:bertini-krull}
The Bertini-Krull theorem is a very explicit {\it iff} criterion for a polynomial $F(\underline t,\underline x)$ 
as above to be generically irreducible. We refer to \cite[theorem 37]{schinzel-book} for the precise 
statement. We recall below two applications (\S \ref {sec:s=1} and \S \ref{sssec:BoDeNa1}).


\subsubsection{Pencil of two polynomials} \label{sec:s=1}
%
%
%
%
%
%
%
%
Assume further that $s=1$, that is:
\vskip 1,5mm

\noindent
\centerline{$F(t,\xbar) = P(\xbar) - t Q(\xbar)$}
\vskip 1mm

\noindent
with $\max(\deg(P),\deg(Q)) \geq 1$. The Bertini-Krull theorem relates the generic irreducibility of $F$ to the 
indecomposability of the rational function $P/Q$. Recall that $P/Q$ is said to be
\defi{decomposable} in $k(\xbar)$ if there exist $A,B \in k[\xbar]$, $B\not=0$ and $h,g \in k[u]$ with $g\not=0$ and 
$\max(\deg g,\deg h) \ge 2$ such that

\vskip 1mm

\noindent
\centerline{$\displaystyle \frac{P}{Q} = \frac{h}{g} \left( \frac{A}{B} \right)$.} \hskip -50mm \footnote{There is also a notion of decomposability for rational functions in one indeterminate. 
Definitions, problems, tools and techniques are however different although there is a Hilbert like 
specialization theorem proved in \cite{BoCheDe} which provides a bridge between indecomposable
polynomials in several indeterminates and those in one indeterminate.}

\noindent
Then we have 
\vskip 1mm

\noindent
(1) {\it $F(t,\xbar) = P(\xbar)-tQ(\xbar)$ is generically irreducible 
if and only if $P/Q$ is indecomposable.}
\vskip 1mm

\noindent
In the special ``polynomial situation'', {\it i.e.} $Q=1$, the condition ``$P$ indecomposable in $k(\underline x)$''
is equivalent to ``$P$ indecomposable in $k[\underline x]$'', {\it i.e.} $P$ does not write $P(\xbar) = h(A(\xbar))$ with $A \in k[\xbar]$ and $h \in k[u]$ with $\deg h \ge 2$. This follows from results due to Gordan and Noether in characteristic $0$ and Igusa and Schinzel in general \cite[theorems 3 and 4]{schinzel-book}.

\vskip 1mm

%


Many articles have been devoted to the spectrum ${\rm sp}(P-tQ)$ in the indecomposable situation. 
We briefly recall the main questions and results in the next paragraphs.
\vskip 2mm

\paragraph{\hbox{\it On the cardinality}}
The general Bertini-Noether bound (6) from \S \ref{intro} gives this inequality, when $k$ 
is of characteristic $0$ or $p>d(d-1)$:
\vskip 1,5mm

\noindent
(2) \hskip 15mm ${\rm card}({\rm sp}(P-tQ)) \leq d^2-1$ \hskip 2mm with $d = \max(\deg P,\deg Q)$.
\vskip 1,5mm

\noindent
We refer to \cite{arnaud-israel} for more details on this bound in the missing cases.

Ruppert shows further in \cite{ruppert} that the Hesse cubic pencil
\vskip 1mm
\centerline{$P(x,y) - t Q(x,y) = x^3+y^3+(1+x+y)^3 - 3t\hskip 1pt {xy(1+x+y)}$}
\vskip 1mm

\noindent
reaches the maximal possible value for $d=3$ if spectral values are 
counted with multiplicity.
That is, the following is checked: 
$\sp(P-t Q)= \{1,j,j^2,\infty\}$ (where $1,j,j^2$ are the cubic roots of $1$)\footnote{The value $\infty$
can be moved to a finite point by a change of coordinates.}; for each $t\point \in \sp(P-t Q)$, 
the curve $P(x,y)-t\point Q(x,y)=0$ breaks into $3$ lines; if the multiplicity 
$2=3-1$ is affected to each of the elements of $\sp(P-t\point Q)$, then $4\times 2 = 8 = d^2-1$.
Furthermore Nguyen asserts \cite{nguyen}  that if the spectrum is counted with multiplicities as above, 
the Hesse cubic pencil is the only example that reaches 
the extremal value $d^2-1$ (for any $d \ge 3$). 


%

%
%
%
%

Another interesting statement from \cite{nguyen} is that when $k=\Cc$, if $P(x,y)=0$ and $Q(x,y)=0$ 
are smooth plane curves, then 
\vskip 1mm

\centerline{${\rm card}({\rm sp}(P-t Q)) \le 3d-3$}
 
\vskip 1mm

\noindent
Furthermore, if $F=x(y^{d-1}-1)+\lambda(x^{d-1}-1)+\mu(x^{d-1}-y^{d-1})$ (with $d\ge3$) is viewed as a 
polynomial in $x,y$, parametrized by $(\lambda,\mu)$ in the plane $k^2$, and $P-tQ$ is obtained by 
restricting $(\lambda,\mu)$ in some $t$-line, then, generically, the pencil $P-tQ$ realizes the bound $3d-3$. 
\vskip 0,5mm

A major result about this issue remains {\it Stein's theorem} for polynomials:
\vskip 1mm

\noindent
(3) {\it If $P\in k[\underline x]$ is indecomposable, then ${\rm card}(\sp(P-t)) < \deg(P)$}.
\vskip 1mm

\noindent
This was first established by Stein \cite{stein} in two variables and in characteristic $0$,
then extended to all characteristics by Lorenzini \cite{lorenzini} and finally generalized
to more variables by Najib \cite{salah_lorenzini}.

%
%
%

\vskip 2mm

\paragraph{\hbox{\it On the spectrum itself}} It can be shown that generically a polynomial 
$P(\underline x)$ is indecomposable and for $\deg(P)>2$ or $\ell>2$, the spectrum $\sp(P-t)$ is empty 
\cite[proposition 2.2]{BoDeNa2}. The question arises then as to whether other finite sets
occur as spectra (within Stein's limitations). Najib \cite{najib1} answers positively to this question. He 
shows that 
\vskip 1mm

\noindent
(4) {\it for any finite subset ${\mathcal S}\subset k$, there exists an indecomposable
polynomial $P\in k[\underline x]$ such that ${\rm sp}(P(\underline x) - t) = {\mathcal S}$.}
\vskip 1mm

\noindent
He can further fix in advance all but one of the irreducible factors of the polynomials 
$P(\underline x) - t^\ast$ with $t^\ast \in {\mathcal S}$ and arrange for Stein's inequality
to be an equality for $P$. For example, for every degree $d\geq 2$ and given $d-1$ points 
$t_1\point,\ldots,t_{d-1}\point \in k$,
he can construct an indecomposable polynomial $P(x,y)$ of degree $d$ such that 
$P(x,y)-t_i\point$ is divisible by $x-t_i\point$, $i=1,\ldots,d-1$. Due to Stein's inequality, 
this already implies that ${\mathcal S}$ exactly equals $\{t_1\point,\ldots,t_{d-1}\point\}$.
Such a polynomial $P(x,y)$ can be made explicit: take
\vskip 1mm

\centerline{$P(x,y) = y(x-t_1\point)(x-t_2\point)\ldots(x-t_{d-1}\point)+x$.}
\vskip 1mm

\noindent
Furthermore some converse is stated in  \cite{nguyen}: a degree $d$ polynomial in two variables 
with a spectrum of cardinality $d-1$ is, up to some change of variables, the polynomial $P$ above.

We end this discussion with open questions.
\vskip 1mm

\noindent
{\bf Problem.} {\it Given an integer $d\geq 1$ and a degree $d$ polynomial $Q\in k[\underline x]$,
\vskip 0,5mm

\noindent
{\rm (a)} find a polynomial $P\in k[\underline x]$ of degree $\leq d$ such that $P/Q$ 
is indecomposable and for which the Bertini-Noether inequality {\rm (2)} from \S \ref{sec:s=1} is an equality.
\vskip 0,5mm

\noindent
{\rm (b)} given a finite subset ${\mathcal S}\subset k$ of cardinality $\leq d^2-1$, find a polynomial 
$P\in k[\underline x]$ such that $\deg(P) \leq d$, $P/Q$ indecomposable and
${\rm sp}(P(\underline x) - tQ(\underline x)) = {\mathcal S}$.}
\vskip 1mm

For (b), the method from \cite{najib1} generalizes to construct a polynomial 
$P$ such that $\deg(P) \leq d$, $P/Q$ indecomposable and ${\rm sp}(P(\underline x) - tQ(\underline x))$ contains any prescribed subset ${\mathcal S}$ of $d-1$ elements. However
the Bertini-Noether bound (2) for rational functions is not sharp enough (as is the Stein bound (3)
for polynomials)
to conclude that the containment is an equality.

\subsubsection{Deformation by monomials} \label{sssec:BoDeNa1}
\label{sec:defmono}
Consider the general case 
\vskip 1mm

\centerline{$F(\tbar,\xbar) = P(\xbar) - t_1 Q_1(\xbar)-\cdots- t_s Q_s(\xbar)$}
\vskip 1mm

\noindent
of a deformation by $s$ polynomials but assume that $Q_1,\ldots,Q_s$ are monomials.
\cite{BoDeNa1} shows how to handle this situation with the Bertini-Krull theorem. The following statements are two selected generic irreducibility criteria from \cite{BoDeNa1}, which can be compared 
to the results of next subsection 
given by the more combinatorial  approach.


\begin{theorem}
\label{thm:typical}
Let $F(t,\xbar) = P(\xbar) - t Q(\xbar)$ with
$P\in k[\underline x]$ of degree $d\geq 1$.
Assume that 
 \vskip 0,5mm

\noindent
{\rm (a)} $Q$ is a monomial of degree $\le d$ and is relatively prime to $P$,
 \vskip 0,5mm

\noindent
{\rm (b)}  $\Gamma(P) \cup \Gamma(Q)$ is not contained in a line,
 \vskip 0,5mm

\noindent
{\rm (c)} $Q$ is not a pure power {\rm (}if $Q(\xbar) = 
a  x_1^{k_1}\cdots x_\ell^{k_\ell}$ then $\gcd(k_1,\ldots,k_\ell)=1${\rm )}.
 \vskip 0,5mm

\noindent
Then $F(t,\xbar)$ is generically irreducible.
\end{theorem}

For example if $P\notin k[x_1]$ and is not divisible by $x_1$,
then $P(\underline x)-t x_1$ is generically irreducible.

%

\begin{theorem}
\label{thm:typical2}
Let $P \in k[\xbar]$ be a polynomial of degree $d\geq 1$ and $Q_1,\ldots,Q_s$ 
be monomials of degree
  $\le d$. Assume further that
 \vskip 0,5mm

\noindent 
{\rm (a)} $s\ge 2$ and  $P,Q_1,\ldots,Q_s$ are relatively prime,
 \vskip 0,5mm

\noindent
{\rm (b)} $\Gamma(P)\cup\Gamma(Q_1)\cup\ldots\cup\Gamma(Q_s)$ is not contained in a line,
 \vskip 0,5mm

\noindent
{\rm (c)} if $\charac(k)=p>0$, at least one of 
  $P,Q_1,\ldots,Q_s$ is not a $p$-th power.
 \vskip 0,5mm

\noindent
Then $P(\underline x)-t_1 Q_1(\underline x)-\cdots - t_s Q_s(\underline x)$ is generically irreducible.
\end{theorem}

For example, in characteristic $0$, for each $k\in \{1,\ldots,d\}$, the polynomial 
$P(x_1,\ldots,x_\ell)+t_1 x_1^k+\cdots + t_\ell x_\ell^k$ 
is generically irreducible.

\subsection{Applications of Gao's criteria}
\label{sec:monomial}
%
The situation is that of polynomials
\vskip 1mm

\centerline{$F(t,\xbar) = P(\xbar) - t \hskip 1pt Q(\xbar)$ with $P,Q \in k[\underline x]$.}
\subsubsection{Gao's first criterion} Gao gives this {\it iff} condition for a polyhedron 
to be non summable \cite{gao} (which he explains to be a generalization of the 
Eisenstein criterion).

\begin{theorem}[Gao] \label{thm:gao}
Let $P \in k[\xbar]$ such that 
\vskip 1mm

\noindent
{\rm (*)} $\Gamma(P)$ is contained in a hyperplane $H\subset \Rr^\ell$ not passing through the origin. 
\vskip 1mm

\noindent
Denote the vertices of $\Gamma(P)$ by $v_1,\ldots,v_k$.
Then $\Gamma_0(P)$ is not summable if and only if the coordinates of $v_1,\ldots,v_k$ are 
relatively prime.
\end{theorem}

This result leads to this generic irreducibility criterion.

\begin{corollary} \label{cor:Gao}
Let $P \in k[\xbar]$ satisfying {\rm (*)} above and $Q \in k[\xbar]$ such that 
\vskip 1mm

\noindent
{\rm (a)} the coordinates of all the vertices of $\Gamma(P)$ are relatively prime,
\vskip 0,5mm

\noindent
{\rm (b)} $Q(0,\ldots,0)\not=0$, $\Gamma(Q)\subset \Gamma_0(P)$ and no monomial of $Q$ is \hbox{a vertex of $\Gamma(P)$.}
\vskip 1mm

\noindent
Then the polynomial $F=P-tQ$ is generically irreducible and even has this stronger property: $P-t^\ast Q$ is irreducible in $k[\xbar]$ for every $t^\ast\in k\setminus\{0\}$.
\end{corollary}

\begin{proof}
Assumptions (a) and (*) and Gao's theorem \ref{thm:gao} show that $\Gamma_0(P)$ is not summable. As $\Gamma(P-tQ) = \Gamma_0(P)$, Minkowski's theorem concludes that $P-tQ$ is generically irreducible and has the stronger property.
\end{proof}

To completely determine ${\rm sp}(P-tQ)$, it remains to decide whether $P$ is irreducible or not. 
Both may happen: for $P(x,y) = x^p - y^q$ with $\gcd(p,q)=1$, we have $\sp(P-t)=\emptyset$
while for $P(x,y) = x^p - xy^q$, $\sp(P-t)= \{0\}$.

The special case of corollary \ref{cor:Gao} for which $Q=1$ yields this conclusion: if $P$ satisfies conditions (*) 
and (a), then $P(\underline x)- t$ is generically irreducible and has the stronger irreducibility property. For all 
positive and relatively prime integers $a$, $b$, $c$, $d$, $P(x,y) = x^ay^b + x^cy^d$ is
such a polynomial.

%

%
%
%
%

\begin{example} Here is one further example where a polynomial $P$ is deformed by a 
polynomial $Q$ ``below'' $P$ to obtain a polynomial $P-tQ$ satisfying the strong generic 
irreducibility property.
%
%
Let $P(x,y) = x^ay^b + x^cy^d$ with $\gcd(a,b,c,d)=1$ and $Q(x,y)\in k[x,y]$ be such that: 
(i) $Q(0,0) \neq 0$ (ii) $\Gamma(Q) \subset \Gamma_0(P)$ and (iii) 
the monomials $x^ay^b$, $x^cy^d$ are not in $\Gamma(Q)$.
Then we have $\sp(P-t Q) \subset \{0\}$.

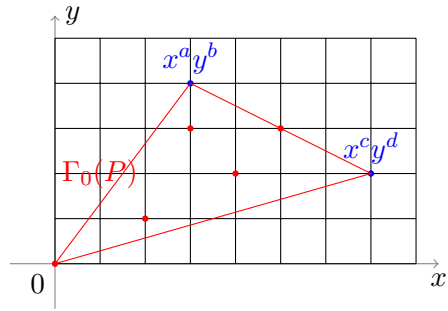
\begin{center}
\begin{tikzpicture}[scale=0.6]

  \draw[gray,->] (-1,0) -- (8.5,0) node[below,black] {$x$};
  \draw[gray,->] (0,-1) -- (0,5.5) node[right,black] {$y$};

 \draw (0,0) grid (8,5);

  \coordinate (O) at (0,0);
  \coordinate (P1) at (3,4);
  \coordinate (P2) at (7,2);
  \coordinate (P3) at (4,2);
  \coordinate (P4) at (3,3);
  \coordinate (P5) at (2,1);
  \coordinate (P6) at (5,3);
  \coordinate (P7) at (4,2);


     \fill[blue] (P1) circle (2pt) node[above]{$x^ay^b$};
     \fill[blue] (P2) circle (2pt) node[above]{$x^cy^d$};
     \fill[red] (P3) circle (2pt);
     \fill[red] (P4) circle (2pt);
     \fill[red] (P5) circle (2pt);
     \fill[red] (P6) circle (2pt);


 \fill[red] (O) circle (2pt);
\node[below left] at (O) {$0$};

 \draw[red] (O)--(P1)--(P2)--cycle;
\node[red] at (1,2) {$\Gamma_0(P)$};
\end{tikzpicture}
\end{center}

\end{example}

\subsubsection{Gao's second criterion}
The following result of \cite{gao} makes it possible to construct non summable polytopes, by induction on the dimension.
\begin{theorem}[Gao] \label{thm:gao2}
Let $P \in k[\xbar]$ such that $\Gamma(P)$ is not summable, has at least two points, 
and is contained in a hyperplane $H$ of $\Rr^\ell$.
Suppose that $Q \in k[\xbar]$ is a polynomial such that $\Gamma(Q)$ is not included in $H$.
Moreover suppose that there exists some monomial $m\in k[\xbar]$  
such that $\Gamma(Q) \subset \Gamma(P+m)$.
Then $\Gamma(P+Q)$ is not summable.
\end{theorem}

With theorem \ref{thm:gao2} one can deform a polynomial $P$ not only by a polynomial $Q$ 
``below'' $P$, but also by some polynomial $Q$ ``above'' $P$.

\begin{example} \label{example:Gao2-1}
Let $P(x,y) = x^p+y^q$ with $p$, $q$ relatively prime. From theorem \ref{thm:gao2} for $Q(x,y) = \sum_{\frac ip+\frac jq < 1} a_{ij}x^iy^j$ (whose monomials are below those of $P$) and $m=1$, the polynomial $P(x,y) - t Q(x,y)$ has empty spectrum. Now take a polynomial $Q$ such that for some $(u,v)\in \Nn^2$, the monomials of $Q$ lie in the triangle 
$(p,0),(0,q),(u,v)$ (and so are above those of $P$) and are distinct from
  $(p,0)$ and $(0,q)$. From theorem \ref{thm:gao2} with $m = x^uy^v$, 
  the spectrum of $P(x,y) - t Q(x,y)$ is empty.  
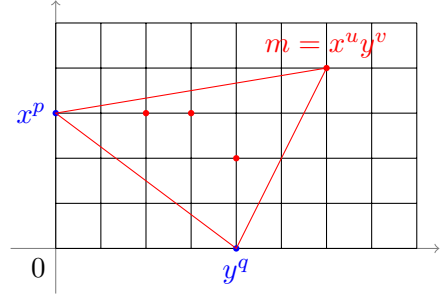
\begin{center}
\begin{tikzpicture}[scale=0.6]

  \draw[gray,->] (-1,0) -- (8.5,0) node[below,black] {};
  \draw[gray,->] (0,-1) -- (0,5.5) node[right,black] {};

 \draw (0,0) grid (8,5);

  \coordinate (O) at (0,0);
  \coordinate (P1) at (0,3);
  \coordinate (P2) at (4,0);
  \coordinate (P3) at (4,2);
  \coordinate (P4) at (3,3);
  \coordinate (P5) at (2,3);
  \coordinate (P6) at (6,4);
  \coordinate (P7) at (4,2);


     \fill[blue] (P1) circle (2pt) node[left]{$x^p$};
     \fill[blue] (P2) circle (2pt) node[below]{$y^q$};
     \fill[red] (P3) circle (2pt);
     \fill[red] (P4) circle (2pt);
     \fill[red] (P5) circle (2pt);
     \fill[red] (P6) circle (2pt) node[above]{$m = x^uy^v$};


\node[below left] at (O) {$0$};

 \draw[red] (P1)--(P2)--(P6)--cycle;
\end{tikzpicture}
\end{center}  
Similar examples can be given in higher dimension starting with $P(x_1,\ldots, x_\ell)=x_1^{p_1}+\cdots+x_\ell^{p_\ell}$ with $p_1,\dots,p_\ell$ relatively prime.
\end{example}

Theorem \ref{thm:gao2} can also be used to explicitly produce a deformation of a (possibly reducible) polynomial $Q(\xbar)$ into an irreducible one.

\begin{example} \label{example:Gao2-2}
Let $Q(x,y)$ be any polynomial and  $P(x,y) = x^p+y^q$ with $p$, $q$ relatively prime and $p,q > \deg Q$. 
Then the polynomial $(x^p+y^q)-tQ(x,y)$ has the strong generic irreducibility property, so  
$Q(x,y)+\mu(x^p+y^q)$ is irreducible in $k[x,y]$ for every $\mu\in k$, $\mu\neq0$.
\end{example}



\bibliography{fampol4}
\bibliographystyle{alpha}

\end{document}